\documentclass[12pt,twoside]{amsart}
\usepackage[latin1]{inputenc}
\usepackage[OT1]{fontenc}
\usepackage[a4paper]{geometry}
\usepackage{amssymb}
\usepackage[all]{xy}

\newtheorem{thm}{Theorem}[section]
\newtheorem{prop}[thm]{Proposition}

\newtheorem{lem}[thm]{Lemma}
\newtheorem{cor}[thm]{Corollary}

\theoremstyle{definition}

\newtheorem{remark}[thm]{Remark}

\numberwithin{equation}{section}

\newcommand{\Aut}{{\mathsf{Aut}}}
\newcommand{\D}{{\mathsf{D}}}
\newcommand{\M}{{\mathsf{M}}}
\newcommand{\OO}{\mathsf{O}}
\newcommand{\DAut}[1]{{\D\Aut(#1)}{}}        
\newcommand{\DAutt}[2]{\D\Aut #1(#2){}}      
\newcommand{\KAut}[1]{{\OO(K(#1))}}
\newcommand{\KAutt}[1]{{\OO_{\FM}(K(#1))}}

\newcommand{\Pic}{{\rm Pic}}
\newcommand{\End}{{\rm End}}
\newcommand{\Hom}{{\rm Hom}}
\newcommand{\Ext}{{\rm Ext}}

\newcommand{\sod}[1]{{\langle #1 \rangle}}   

\newcommand{\embed}{\,\hookrightarrow}

\newcommand{\isom}{ \text{{\hspace{0.48em}\raisebox{0.8ex}{${\scriptscriptstyle\sim}$}}}
                    \hspace{-0.65em}{\rightarrow}\hspace{0.3em}}

\newcommand{\coloneq}{\mathrel{\mathop:}=}
\newcommand{\eqcolon}{=\mathrel{\mathop:}}

   \newcommand{\ko}{{\mathcal O}}
\newcommand{\kp}{{\mathcal P}}

 \newcommand{\IC}{\mathbb{C}}  
   \newcommand{\IP}{\mathbb{P}}
\newcommand{\IQ}{\mathbb{Q}}  \newcommand{\IZ}{\mathbb{Z}}

\DeclareMathOperator{\cone}{cone}
\newcommand{\cat}[1]{\begin{bf}#1\end{bf}}
\newcommand{\coh}{{\cat{Coh}}}

\newcommand{\FM}{\mathsf{FM}}

\newcommand{\TT}{\mathsf{T}\!}
\newcommand{\FFF}{\mathsf{F}\hspace{-2pt}}
\newcommand{\PP}{\mathsf{P}}
\newcommand{\PPP}{\mathsf{P}\!}

\newcommand{\eq}{\text{eq}}
\newcommand{\id}{{\rm id}}
\newcommand{\tr}{{\rm tr}}
\newcommand{\dual}{^{\vee}}
\newcommand{\orth}{^\perp}

\newcommand{\kk}{\mathbf{k}}
\newcommand{\Sn}{{S_n}}
\newcommand{\Xn}{{X^{[n]}}}
\newcommand{\tX}{{\tilde X}}
\newcommand{\tZ}{{\tilde Z}}
\newcommand{\tE}{{\tilde E}}
\newcommand{\tF}{{\tilde F}}
\newcommand{\tPhi}{{\tilde \Phi}}

\newcommand{\deltan}{{\delta_{[n]}}}

\begin{document}

\title[New autoequivalences]{On autoequivalences of some Calabi--Yau and hyperk\"ahler varieties}
\author[Ploog]{David Ploog}
\address{David Ploog, Leibniz Universit\"at Hannover\\
Institut f\"ur Algebraische Geometrie\\
Welfengarten 1\\
30167 Hannover, Germany}
\email{ploog@math.uni-hannover.de}

\author[Sosna]{Pawel Sosna$^1$}
\address{Pawel Sosna, Fachbereich Mathematik der Universit\"at Hamburg\\
Bundesstra\ss e 55\\
20146 Hamburg, Germany}
\email{pawel.sosna@math.uni-hamburg.de}

\thanks{$^1$ Supported by the RTG 1670 of the  DFG (German Research Foundation)}

\begin{abstract}
We show how one can construct new autoequivalences of some Calabi--Yau and hyperk\"ahler varieties by passing from a smooth projective surface to the associated Hilbert scheme of points.
\end{abstract}
\maketitle

\setcounter{tocdepth}{1}
\tableofcontents

\section{Introduction}

The bounded derived category of coherent sheaves on a smooth projective variety $Z$ has attracted a lot of attention in the last decade, partly because of its relation to symplectic geometry via the Homological Mirror Symmetry conjecture \cite{Kontsevich}. In particular, it is natural to study its group of autoequivalences. By classical results of Bondal and Orlov \cite{Bondal-Orlov}, this group is completely understood if the canonical bundle or its dual is ample: the group is then generated by line bundle twists, automorphisms of the variety and the shift. On the other hand, varieties with trivial canonical bundle always have autoequivalences beyond the above mentioned standard ones. For Calabi--Yau varieties, there are spherical twists \cite{Seidel-Thomas}, and for hyperk\"ahler varieties, there are $\IP^n$-twists \cite{Huybrechts-Thomas}. Both of these tear apart the standard heart of the derived category. In general, there are few general constructions of \emph{exotic} autoequivalences, that is, non-standard autoequivalences which are neither spherical nor $\IP^n$-twists, but see, for instance, Anno's spherical functors \cite{Anno} and Addington's examples \cite{Addington}.

In this article, we present a way to turn autoequivalences on K3 and Enriques surfaces into interesting autoequivalences of varieties with trivial canonical bundle, namely Hilbert schemes of points on K3 surfaces and Calabi-Yau varieties made from Enriques surfaces. Our main result can be roughly stated as follows, see Theorems \ref{thm:enriques} and \ref{thm:hilb} for details.
\smallskip

\noindent
{\textbf{Main result.}} \emph{Let $Z$ be either the Hilbert scheme of $n$ points on a smooth projective K3 surface or the Calabi--Yau variety which is the canonical cover of the Hilbert scheme of $n$ points on an Enriques surface. Then $Z$ admits exotic autoequivalences.}

\smallskip
The two constructions share the following technical tools: we use spherical twists on a K3 surface, linearisations for finite group actions and the induction construction from \cite{Ploog}. Regarding the latter, we provide a much simpler description of the induced functor in Proposition~\ref{prop:induction}, and we say something about the image on the level of K-groups in Proposition~\ref{prop:K}.

It seems to be a difficult and open problem to decide whether a given autoequivalence is in the subgroup generated by the standard autoequivalences and all spherical or $\IP^n$-twists, respectively. In fact, we are unable to check this for our new autoequivalences.
\smallskip 

\noindent
\textbf{Acknowledgements.} We thank Daniel Huybrechts, Eyal Markman and Richard Thomas for encouragement and support.


\section{Preliminaries}

In this article, all varieties will be smooth, projective, and defined over an
algebraically closed field $\kk$ of characteristic 0. We will denote by $\D(X)$
the bounded derived category of coherent sheaves of a variety $X$, and we
write $\DAut{X}$ for the group of autoequivalences of $\D(X)$.
This is done for brevity, but also to emphasise that we consider an
autoequivalence of $\D(X)$ as a `higher (or derived) symmetry' of $X$.

All categories, and functors between them, are assumed to be $\kk$-linear.
Functors between triangulated categories are assumed to respect the triangulated
structures. We will write distinguished triangles sloppily as $A\to B\to C$,
suppressing the connecting morphism $C\to A[1]$.

We use the same symbol for functors between abelian categories, and their derived functors. For example, we will write $f_*\colon\D(X)\to\D(Y)$ for the derived push-forward functor of a proper map $f\colon X\to Y$. Given a line bundle $L$, we will write $\M_L$ for the autoequivalence given by tensoring with $L$.

\subsection{The varieties}\label{sub:the varieties}
Before we introduce the varieties performing in this article, we have to recall
the notion of canonical covers. Let $Z$ be a variety with torsion canonical
bundle of (minimal) order $k$. The \emph{canonical cover} $\tZ$ of $Z$ is the
unique (up to isomorphism) variety with trivial canonical bundle and an \'etale
morphism $\pi\colon \tZ \to Z$ of degree $k$ such that
$\pi_*\ko_\tZ=\bigoplus_{i=0}^{k-1}\omega_Z^i$. In this case, there is a free
action of the cyclic group $\IZ/k\IZ$ on $\tZ$ such that $\pi$ is the quotient
morphism.

\begin{enumerate}
\item A \emph{K3 surface} is defined by having trivial canonical bundle and no algebraic 1-forms. K3 surfaces have a rich group of autoequivalences, see, for example, \cite{HMS} or \cite{Bri2}.
\item An \emph{Enriques surface} is by definition the quotient of a K3 surface by a fixed point free involution. It follows that the canonical bundle has order two and that the canonical cover of an Enriques surface is a K3 surface with a free involution.
\item Given a K3 surface $X$, the \emph{Hilbert scheme} $\Xn$ (of $n$ points) is a prime example for an irreducible symplectic (also called hyperk\"ahler) variety, i.e.\ there is a no-where vanishing 2-form, unique up to constants.
\item Given an Enriques surface $Y$, the Hilbert scheme $Y^{[n]}$ has order two canonical bundle. The canonical cover of $Y^{[n]}$ is a Calabi--Yau variety, that is, its canonical bundle is trivial and its structure sheaf has no cohomology apart from lowest and top degrees, see \cite{Oguiso-Schroer} or \cite[Prop.\ 1.6]{Nieper}.
\end{enumerate}

\subsection{Varieties with group actions} \label{sub:linearisations}
Let $X$ be a variety on which a finite group $G$ acts. A natural object to study in this setting is the quotient stack $[X/G]$. If $G$ acts freely, then $[X/G]$ coincides with the variety quotient. However, if the action has fixed points, then the quotient stack is better behaved than the quotient variety. For example, $[X/G]$ is always a smooth Deligne--Mumford stack, making it as good as smooth, projective varieties from the point of view of derived categories.

The category of coherent sheaves on $[X/G]$ has a nice and well-known description using linearised sheaves. We recall some aspects of this theory. A coherent sheaf $A$ is called \emph{invariant} if $g_*A\cong A$ for all $g\in G$. This is not a very good notion, however, as the full subcategory of invariant sheaves is not abelian. It is better to make the choice of isomorphisms part of the data: a $G$-linearisation on $A$ is given by a collection of isomorphisms $\alpha_g\colon A\isom g_*A$ for all $g\in G$ such that $\alpha_1=\id_A$ and the natural cocycle condition holds. The \emph{equivariant category} $\coh^G(X)$ has as objects linearised sheaves, $(A,\alpha)$; the morphisms are given by
 $\Hom_{\coh^G(X)}((A,\alpha),(B,\beta))\coloneq\Hom_{\coh(X)}(A,B)^G$
where the linearisations $\alpha$ and $\beta$ turn $\Hom_X(A,B)$ into a $G$-representation in the natural way. The category $\coh^G(X)$ is abelian (see \cite{Groth}). Moreover, there is a tautological equivalence of abelian categories $\coh^G(X)=\coh([X/G])$ and, in turn, an equivalence of triangulated categories $\D(\coh^G(X))=:\D^G(X)=\D([X/G])$. As a remark on notation, we will generally write $\D^G(X)$ instead of $\D([X/G])$ but $\DAut{[X/G]}$ instead of $\Aut(\D^G(X))$.

Obviously, every linearised sheaf (or object) is invariant. The reverse question is subtle: an invariant object may have none, a unique, or many different linearisations. The situation is easier for invariant objects $A$ which are simple, i.e.\ $\Hom(A,A)=\kk$: such an $A$ has a group cohomology class $[A]\in H^2(G;\kk^*)$ where $G$ acts trivially on $\kk^*$. If $[A]=0$ in $H^2(G;\kk^*)$, then $A$ is linearisable, and the set of $G$-linearisations is a torsor over $H^1(G;\kk^*)=\Hom(G,\kk^*)$ (for these facts see \cite[Lemma~1]{Ploog}). The group cohomology $H^2(G;\kk^*)$ is called the \emph{Schur multiplier}, it is a finite abelian group. (This is a place where we need $\kk$ to be algebraically closed and of characteristic 0.) We will only have need for cyclic and symmetric groups. For these, we have
\begin{align*}
H^2(\IZ/n\IZ;\kk^*) &= 0, &\qquad
H^2(\Sn;\kk^*) &= \begin{cases} 0 & n=2,3,\\ \IZ/2\IZ & n\geq4 \end{cases} \\
H^1(\IZ/n\IZ;\kk^*) &= \IZ/n\IZ, &\qquad
H^1(\Sn;\kk^*) &= \IZ/2\IZ .
\end{align*}

\subsection{Autoequivalences of Hilbert schemes} \label{sub:Hilbert}

Let $X$ be a surface. There is a canonical and faithful method of turning autoequivalences for $X$ into autoequivalences for the Hilbert scheme $\Xn$:
\[ \DAut{X} \embed \DAut{\Xn} .\]

We proceed to explain the method. First, by Haiman's well-known result \cite{Haiman}, the Hilbert scheme $\Xn$ is isomorphic to the $\Sn$-equivariant Hilbert scheme $\Sn\text{-Hilb}(X^n)$, where $\Sn$ acts on $X^n$ by permuting the factors. This induces an equivalence of abelian categories $\coh(\Xn)\cong\coh(\Sn\text{-Hilb}(X^n))$ and in turn an equivalence of triangulated categories $\D(\Xn)\cong\D(\Sn\text{-Hilb}(X^n))$.

Second, we employ the derived McKay correspondence of \cite{BKR}, which gives an equivalence $\Phi\colon\D(\Sn\text{-Hilb}(X^n))\isom\D^\Sn(X^n)$. We will only apply these results for K3 surfaces $X$, so this is a genuinely symplectic situation. However, the BKR equivalence actually holds for arbitrary $X$ (smooth and projective, as always). Altogether, we obtain an equivalence $\D(\Xn)\cong\D^\Sn(X^n)$ which will allow us to bring equivariant methods to bear on the problem.

The standard calculus of Fourier--Mukai kernels extends to the equivariant
setting: given two varieties $Z_1$ and $Z_2$ with finite group actions by $G_1$
and $G_2$, respectively, then objects in $\D^{G_1\times G_2}(Z_1\times Z_2)$
give rise to functors $\D^{G_1}(Z_1)\to\D^{G_2}(Z_2)$ with the additional
change that the projection $\pi_2\colon Z_1\times Z_2\to Z_2$ is $G_1\times
G_2$-equivariant with $G_1$ acting trivially on $Z_2$, and one has to take
$G_1$-invariants after the push-forward $\pi_{2*}$ (see \cite{BKR}). Kawamata's
\cite{Kaw} extends Orlov's famous result that all equivalences are of this type.

We restrict to the case of $G_1=G_2\eqcolon G$. Here, $G\times G$-linearised
kernels in $\D(Z_1\times Z_2)$ can be obtained from diagonally linearised
kernels, see \cite{Ploog}: let $P\in\D^{\Delta G}(Z_1\times Z_2)$ where
$\Delta G$ acts on $Z_1\times Z_2$ via $\Delta\colon G\embed G\times G$,
$g\mapsto(g,g)$. Then the object $G\cdot P\coloneq \bigoplus_{g\in G} (1,g)_*P$
is canonically $G\times G$-linearised (and isomorphic to $\bigoplus_{g\in G}
(g,1)_*P$). The crucial point of this construction is that it respects
equivalences: if $\FM_P\colon\D(Z_1)\isom\D(Z_2)$ is an equivalence of the
ordinary categories, then $\FM_{G\cdot P}\colon\D^G(Z_1)\isom\D^G(Z_2)$ is an
equivalence of the equivariant categories. Furthermore, in the special case of
$Z_1=Z_2\eqcolon Z$, we get a group homomorphism
\[ \DAutt{^{\Delta G}}{Z} \to \DAut{[Z/G]} \]
where $\DAutt{^{\Delta G}}{Z}$ is the set of autoequivalences of $\D(Z)$ whose
Fourier--Mukai kernel is equipped with a $\Delta G$-linearisation. Autoequivalences in
the image of this homomorphism are called \emph{induced autoequivalences}.
However, there is a simpler description of induced functors --- these are the
original Fourier--Mukai transform on the underlying complexes of sheaves:

\begin{prop} \label{prop:induction}
Let $Z_1$ and $Z_2$ be two smooth, projective varieties with actions by a finite group $G$ and let $(P,\mu)\in\D^{\Delta G}(Z_1\times Z_2)$. Then the induced functor
 $\FM_{G\cdot P}\colon\D^G(Z_1)\to\D^G(Z_2)$
is isomorphic to the functor given by
\[ 
   (A,\alpha) \mapsto (\pi_{2*}(P\otimes\pi_1^*A)),
\pi_{2*}(\mu \otimes \pi_1^*\alpha))
    = (\FM_P(A), \pi_{2*}(\mu \otimes \pi_1^*\alpha)) .\]
\end{prop}

\begin{proof}
The essential point is to determine the $G\times1$-invariants of
$\pi_{2*}(G\cdot P\otimes\pi_1^*A)$. By definition, $(g,1)$ acts on $G\cdot P$
by a permutation of summands which we denote $\sigma_g$, and on $\pi_1^*A$ as
$\pi_1^*\alpha_g$. Therefore, $(g,1)$ is the automorphism
$\pi_{2*}(\sigma_g\otimes\pi_1^*\alpha_g)$ of the object $\pi_{2*}(G\cdot
P\otimes\pi_1^*A)$. The diagonal map
$\pi_{2*}(P\otimes\pi_1^*A)\to\pi_{2*}(G\cdot P\otimes\pi_1^*A)$ induced on
each summand by the canonical isomorphism
 $\pi_{2*}(P\otimes\pi_1^*A) \isom \pi_{2*}((g,1)_*P\otimes\pi_1^*A)$
is an isomorphism
$\pi_{2*}(P\otimes\pi_1^*A)\isom(\pi_{2*}(G\cdot P\otimes\pi_1^*A))^{G\times
1}$. The two $G$-linearisations on this object coincide.
\end{proof}

Starting with a surface $X$, we observe that an autoequivalence
$\varphi\in\DAut{X}$ brings about another one on the $n$-fold product $X^n$,
viz.\ $\varphi^n\in\DAut{X^n}$. On the level of Fourier--Mukai kernels, if
$\varphi=\FM_P$ with $P\in\D(X\times X)$, then $\varphi^n=\FM_{P^{\boxtimes
n}}$, where $P^{\boxtimes n}\in\D(X^n\times X^n)$ is the $n$-fold exterior
tensor product. This kernel is diagonally $\Sn$-linearised in a canonical
fashion, although the linearisation is not unique (see
Section~\ref{sub:linearisations}). Combining this association with the one
from the previous paragraph, we obtain two group homomorphisms
\begin{equation}\label{liftingtohilb}
\deltan\colon \DAut{X} \to \DAutt{^{\Delta\Sn}}{X^n} \to \DAut{[X^n\!\!\,/\!\Sn]} = \DAut{\Xn}
\end{equation}
whose composition, in the following denoted by $\varphi\mapsto\varphi^{[n]}$,
is shown to be injective in \cite{Ploog}.

In the rest of this paragraph, we study the effect of this construction on 
the level of K-groups. An autoequivalence $\varphi$ of $\D(X)$ induces an isometry
of $K(X)$, denoted $\varphi^K$ (using the Euler form as a pairing, in general 
degenerate and not symmetric). However, it is not clear that
 $\deltan\colon\DAut{X}\to\DAut{\Xn}$
induces a map $\KAut{X}\to\KAut{\Xn}$. Another issue is that many non-trivial
autoequivalences of $\D(X)$ induce the identity map on K-groups. For example,
the square of a spherical twist or any $\IP^n$-twist (see \ref{sub:single}) acts
trivially on the K-group. Both of these issues are solved by restricting focus
to the image of $\DAut{X}\to\KAut{X}$, which we denote $\KAutt{X}$.

\begin{prop} \label{prop:K}
Let $X$ be a surface with $\chi(\ko_X)\neq 0$ and torsion free K-group.
Then the mapping
 $\deltan^K\colon\KAutt{X}\to\KAut{\Xn}$, $\varphi^K\mapsto (\deltan(\varphi))^K$
is well-defined and injective:
\[\xymatrix{      & \DAut{X}  \ar@{^{(}->}[r]^\deltan    \ar@{->>}[d]  & \DAut{\Xn} \ar[d] \\
         \KAut{X} & \KAutt{X} \ar@{^{(}->}[r]^{\deltan^K} \ar@{_{(}-}[l] & \KAut{\Xn}        .}\]
\end{prop}

\begin{proof}
Let $\varphi_1,\varphi_2\in\DAut{X}$ such that $\varphi_1^K=\varphi_2^K$. We 
want to show that $(\deltan(\varphi_1))^K=(\deltan(\varphi_2))^K$. Only for
notational reasons, we will deal with $n=2$. Now $K(X\times X)$ is generated
by classes of the form $F\boxtimes F'$, compare the proof of \cite[Lem.\ 5.2]{Kuznetsov},
and $K(\D^{S_2}(X\times X))$ is generated by classes of the forms $F\boxtimes F$
and $F\boxtimes F'\oplus F'\boxtimes F$ by the same reasoning for the quotient stack. The first statement implies
 $[\varphi_1^{\times 2}(F\boxtimes F')] = [\varphi_1(F)\boxtimes\varphi_1(F')] =
  [\varphi_1(F)]\boxdot[\varphi_1(F')] = [\varphi_2(F)]\boxdot[\varphi_2(F')]$
and the second allows similar reasoning, invoking Proposition~\ref{prop:induction}.

We consider the following functors:
\begin{align*}
 s  \colon & \D(X) \to \D^{S_n}(X^n), \quad F \mapsto p_1^*F\oplus\cdots\oplus p_n^*F \\
\pi \colon & \D^{S_n}(X^n) \xrightarrow{\text{forget}} \D(X^n) \xrightarrow{p_{1*}} \D(X)
\end{align*}
($s(F)$ has a natural $S_n$-linearisation from permutation of direct summands.)

By the projection formula,
 $\pi s = \id \otimes H^\bullet(\ko_X) \oplus H^\bullet(\cdot) \otimes \ko_X^{n-1}$. 
This implies that $s^K$ is injective: let $F\in\D(X)$ with $[s(F)]=0$. Then also
$[\pi s(F)]=0$ and in thus
\[ 0 = [\pi s(F)] = [F^{\oplus e} \oplus \ko_X^{\oplus(n-1)\chi(F)}]
                  = e[F] + (n-1)\chi(F)[\ko_X] \]
where $e\coloneq\chi(\ko_X)\neq0$ by assumption. Now it is a general algebraic fact 
that the endomorphism 
 $K(X) \to K(X)$, $[F] \mapsto e[F] + (n-1)\chi(F)[\ko_X]$
is injective, which in turn forces $s^K$ to be injective:

Let $A$ be a torsion-free abelian group, $\chi\colon A\to\IZ$ a homomorphism
and $a_0\in A$ a fixed element. Denote by $\chi_\IQ\colon A\otimes \IQ\to \IQ$.
Put $e\coloneq\chi(a_0)$ and consider the endomorphism
$\mu\colon A\to A$, $a\mapsto e a+(n-1)\chi(a)a_0$, where $n>0$ is a natural number.
Then $\mu(a)=0$ implies that $a=\lambda a_0$ is a (rational) multiple of $a_0$. But 
$\mu(\lambda a_0)=e\lambda a_0+(n-1)\chi(\lambda a_0)a_0=\lambda e n a_0 \neq 0$ as long
as $e\neq0$.

Next, we are going to use objects of the form $s(A)\in\D^{S_n}(X^n)$ to show that
$\deltan^K$ is injective on the image of $c$: Let $\varphi\in\DAut{X}$ be an 
autoequivalence such that $\varphi^K\in\KAut{X}$ is not the identity. Therefore, there
exists $F\in\D(X)$ with $\varphi^K([F])=[\varphi(F)]\neq[F]$.

\begin{align*}
\deltan^K(\varphi^K)([s(F)]) &= [\varphi^{[n]}(p_1^*F\oplus\cdots\oplus p_n^*F)] \\
                            &= [p_1^*\varphi(F)\otimes p_2^*\varphi(\ko_X)\otimes\cdots \otimes p_n^*\varphi(\ko_X) \oplus\cdots] \\
                            &= [s'(\varphi(F))]
\end{align*}
where we first use that, for example, $p_1^*F\cong F\boxtimes \ko_X\cdots \boxtimes \ko_X$ (($n-1$)-times) and so forth, and the second identity follows from Proposition~\ref{prop:induction}, and
 $s'\colon \D(X) \to \D^{S_n}(X^n)$, 
 $F' \mapsto \bigoplus_i \big( p_i^*F'\otimes \bigotimes_{j\neq i} p_j^*\varphi(\ko_X) \big)$.
Injectivity of $s'$ is proved like that of $s$.
\end{proof}

\begin{remark}
In particular, this result applies to K3 surfaces, since the Chern character $K(X)\to \text{CH}^*(X)$ is integral and injective, and the Chow ring is torsion-free for a K3 surface by \cite{Roitman}.
\end{remark}

\subsection{Autoequivalences of canonical covers} \label{sub:covers}

Let $Z$ be a variety with torsion canonical bundle, $\tZ$ be its canonical
cover and $\pi\colon \tZ\rightarrow Z$ be the quotient map. The relation between 
autoequivalences on $Z$ and those on the canonical cover $\tZ$ was studied in
\cite{Bridgeland-Maciocia}.

We start with the subgroup $G\embed\DAut{\tZ}$, $g\mapsto g_*$.
An autoequivalence $\tPhi$ of $\D(\tZ)$ is \emph{invariant} (under the conjugation
action of $G$ on $\DAut{\tZ}$) if $g_*\tPhi\cong\tPhi g_*$ for all $g\in G$, and
invariant autoequivalences form a subgroup $\DAut{\tZ}^G\subset\DAut{\tZ}$, which is
just the centraliser of $G$ in $\DAut{\tZ}$.
More generally, $\tPhi$ is called \emph{equivariant}, if there is a group automorphism
$\mu\colon G\isom G$ such that $g_*\tPhi\cong\tPhi\mu(g)_*$ for all $g\in G$. Evidently,
the group automorphism $\mu$ is uniquely determined by $\tPhi$. Therefore, equivariant
autoequivalences form a subgroup $\DAutt{_\eq}{\tZ}$ of $\DAut{\tZ}$ given by the 
semi-direct product $\DAut{\tZ}^G\rtimes\Aut(G)$; this is just the normaliser of $G$ in
$\DAut{\tZ}$. In particular, the subgroups of invariant or equivariant autoequivalences
contain $G$ as a normal subgroup.
We note that for double covers the two notions coincide.

By \cite[\S4]{Bridgeland-Maciocia}, an equivariant functor $\tPhi$ descends, i.e.\
there exists a functor $\Phi\in\DAut{Z}$ with functor isomorphisms
  $\pi_*\circ\tPhi \cong \Phi \circ \pi_*$ and
  $\pi^*\circ\Phi \cong \tPhi \circ\pi^*$.
Two functors $\Phi$, $\Phi'$ that $\tPhi$ descends to differ by a line
bundle twist with a power of $\omega_Z$. As $\omega_Z$ is torsion, the line bundle
twists $\M_{\omega_X}^i$ form a subgroup of $\DAut{Z}$ isomorphic to $G$.
Hence, we get a group homomorphism $\DAut{\tZ}^G\rtimes\Aut(G) \to \DAut{Z}/G$. Note that
the Serre functor $\M_{\omega_Z}[\dim Z]$ commutes with all autoequivalences,
so that $G$ is central in $\DAut{Z}$, hence a normal subgroup.

In the other direction, it is also shown in \cite[\S4]{Bridgeland-Maciocia} that
every autoequivalence of $\D(Z)$ has an equivariant lift. Two lifts differ up to the
action of $G$ (in $\DAut{\tZ}$). This induces a group homomorphism
 $\DAut{Z} \to \DAutt{_\eq}{\tZ}/G$.
Furthermore, if two autoequivalences of $\D(Z)$ lift to the same equivariant
autoequivalence, then they differ by a line bundle twist by a power of $\omega_Z$,
i.e.\ the homomorphism above is a $k:1$-map.

Summing up, we see that autoequivalences of $Z$ and of its canonical cover are related
in the following ways:
\[ \begin{array}{rl @{\qquad\text{with }} l}
 \DAut{Z}          &\xrightarrow[\text{\phantom{descend}}]{\text{lift}} \DAutt{_\eq}{\tZ}/G     & G=\sod{g_*,\id_G}     \\
 \DAutt{_\eq}{\tZ} &\xrightarrow[\text{descend}]{} \DAut{Z}/G                                   & G=\sod{\M_{\omega_Z}} 
\end{array} \]

\subsection{Functors made from a single object} \label{sub:single}

Let $X$ be a $d$-dimensional variety. An object $E\in\D(X)$ is called
\begin{itemize}
\item \emph{spherical} if $E\otimes\omega_X\cong E$ and $\Ext^*(E,E)\cong H^*(S^d,\kk)=\kk\oplus\kk[-d]$ and
\item \emph{$\IP^d$-object} if $E\otimes\omega_X\cong E$ and $\Ext^*(E,E)\cong H^*(\IP^d_\IC,\kk)$ as graded rings.
\end{itemize}
For example, a variety is (strict) Calabi--Yau if and only if its structure sheaf is spherical.
Likewise, a variety is irreducible symplectic if and only if its structure sheaf is a $\IP^d$-object.

To any object $E$, one can associate three functors:
\begin{itemize}
\item $\FFF_E=\Hom^\bullet(E,\cdot)\otimes E$ with FM kernel $E\dual\boxtimes E$;
\item $\TT_E$ with FM kernel $\cone(E\dual\boxtimes E \xrightarrow{\tr} \ko_\Delta)$;
\item $\PPP_E$ is defined by a double cone. Choose a basis $(h_1,\ldots,h_r)$ for $\Hom^2(E,E)$ and set
 $H \coloneq \sum_i h_i\dual\boxtimes\id_E-\id_{E\dual}\boxtimes h_i \colon E\dual\boxtimes
 E[-2] \to E\dual \boxtimes E$. Then $\tr\circ H=0$, so that the trace morphism factors through $\cone(H)$, and $\PPP_E$ is defined by its FM kernel $\cone(\cone(H)\to\ko_\Delta)$.
\end{itemize}
The functors can also be described using triangles
\[ \FFF_E(A)\to A\to \TT_E(A) \]
and
\[ \xymatrix@R=2.5ex@C=4em{
   \FFF_E(A)[-2] \ar[r]^-H \ar@{-->}[dr]_0 & \FFF_E(A) \ar[r] \ar[d]^\tr & \cone(H) \ar[dl] \\
                                           & A \ar[dl] \\
   \PPP_E(A)
} \]
for all objects $A\in\D(X)$.

Obviously, $\FFF_E|_{E\orth}=0$ whereas $\TT_E|_{E\orth}=\id$ and $\PPP_E|_{E\orth}=\id$.
Furthermore, if $E$ is spherical, then $\TT_E(E)=E[1-d]$ and $\TT_E$ is an autoequivalence,
called the \emph{spherical twist} along $E$. Likewise, if $E$ is a $\IP^d$-object, then
$\PPP_E(E)=E[-2d]$ and $\PPP_E$ is an autoequivalence, called the \emph{$\IP^d$-twist} along $E$.

\section{New autoequivalences on CY manifolds arising from Enriques surfaces}

Let $X$ be an Enriques surfaces and $\tX$ the K3 surface arising as canonical
double cover, with projection $\pi\colon\tX\to X$ and deck involution $\tau$.
The maps are related by $\pi^*\pi_*=\id\oplus\tau_*$ as endofunctors of $\D(\tX)$.
The action of $\tau$ is free, so that $X=\tX/\tau=[\tX/\tau]$. (All groups in this
section are of order two, so we will denote quotients and invariants just by
writing the non-trivial group element.)
The Hilbert scheme $X^{[n]}$ has torsion canonical bundle of order two, and
its canonical cover $Z$ is a Calabi--Yau variety. We denote the covering
involution of $Z$ by $\tau_Z$.

We start off with a general result that says how to obtain autoequivalences
of the Calabi--Yau variety $Z$ (up to a 2:1 ambiguity) from invariant
autoequivalences of the K3 surface $\tX$.

\begin{prop}
Any invariant autoequivalence of the K3 surface $\tX$ gives rise to two invariant
autoequivalences on the Calabi--Yau variety $Z$. More precisely, there is an 
injective group homomorphism $\DAut{\tX}^\tau \embed \DAut{Z}^{\tau_Z}/\tau_Z$.
\end{prop}

\begin{proof}
Consider the following diagram:
\[ \xymatrix@C+=3em{
                                        & \DAut{X} \ar[d] \ar[r]^{\text{induce}} & \DAut{\Xn} \ar[d] \ar[r]^{\text{lift}} & \DAut{Z}^{\tau_Z}/\tau_Z \ar@{=}[d] \\
\DAut{\tX}^\tau \ar[r]^-{\text{descend}} & \DAut{X}/\omega_X  \ar@{-->}[r]^? & \DAut{\Xn}/\omega_\Xn \ar@{-->}[r]^? & \DAut{Z}^{\tau_Z}/\tau_Z
} \]
The labelled homomorphisms have been defined in Sections~\ref{sub:Hilbert} and \ref{sub:covers}.
Our aim is to show that induction and lifting factor through the maps marked with question marks.

We consider the induction first. The group actions on $\DAut{X}$ and $\DAut{\Xn}$ are given by
the line bundle twists with the canonical bundles $\omega_X$ and $\omega_\Xn$, respectively. By
construction of the induction, the twist by a line bundle $L$ on $X$, i.e.\ the Fourier--Mukai
kernel $\Delta_{X*}L$ is sent to $S_n\cdot\Delta_{X^n_*}L^{\boxtimes n}$. It follows from
Proposition~\ref{prop:induction} that this autoequivalence is just the line bundle twist with
$L^{[n]}\in\Pic(\Xn)$, where $L^{[n]}$ is the bundle given by the Fourier--Mukai transform with the universal family as kernel. Furthermore, we recall that the derived McKay correspondence
 $\D(\Xn)=\D(S_n\text{-Hilb}(X^n)) \isom \D^{S_n}(X^n)$, with the universal cycle as Fourier--Mukai
kernel, maps the structure sheaf $\ko_\Xn$ to the canonically linearised structure sheaf $\ko_{X^n}$.
This already implies that $\omega_\Xn$ is mapped to $\omega_{X^n}$, the latter with the canonical
linearisation again (because equivalences commute with Serre functors).
Altogether, we find that the induction homomorphism does map the line bundle twist with $\omega_X$
the that of $\omega_\Xn$. Hence the left hand homomorphism on quotients is well defined. It is injective, since we divide out the kernel of the induction map.

Now we turn to the lifting. Here we have to show that lift maps the line bundle twist along
$\omega_\Xn$ to the subgroup generated by push-forward $\tau_{\Xn*}$ of the deck transformation and this is a general property of the lifting. Furthermore, the induced map is injective by the same argument as in the last paragraph.

The statement about injectivity of the composition now follows immediately, since descend is always injective on the subgroup of invariant autoequivalences.
\end{proof}

One can, of course, ask whether this construction gives anything new. For this let $\tE\in\D(\tX)$ be an invariant spherical object on a K3 surface, i.e.\ $\tau_*(\tE)\cong \tE$. The associated twist functor $\TT_\tE$ is equivariant (in the sense of Bridgeland--Maciocia, see Section~\ref{sub:covers}), because by \cite[Lemma 8.21]{Huybrechts} we have
\[ \tau_* \TT_\tE \cong \TT_{\tau_*\tilde{E}} \tau_* \cong \TT_\tE \tau_* \]
Hence it induces two autoequivalences of $X$, which we will denote by $\Phi_i\cong \Phi \circ \M_{\omega}^i$. In particular, choosing $\tE=\ko_\tX$, we obtain two \emph{canonical} autoequivalences
of $X$. Let us note some properties in this case. For an object $A$ in a triangulated category, $\langle A\rangle$ denotes the smallest triangulated subcategory containing it.

\begin{lem}
With notation as above we have $\Phi_i(A)\cong A[-1]$ for any $A \in \pi_*\langle\tE\rangle$. Furthermore, $\Phi_i(B)\cong B$ for any $B \in \pi_*(\tE^\bot)$.
\end{lem}

\begin{proof}
By definition of equivariant functors,
 $\Phi_i \pi_*(\tE)\cong \pi_*\tPhi(\tE)\cong \pi_*\tE[-1]$,
and
 $\Phi_i \pi_*(\tF)\cong \pi_*\tPhi(\tF)\cong \pi_*(\tF)$
for any $\tF\in\tE^\bot$. The claims follow immediately.
\end{proof}

\begin{remark}
Note that $\pi_*(\tE^\bot)$ is not a full subcategory, but nevertheless we have, on objects, that $\pi_*(\tE^\bot)\subset (\pi_*\tE)^\bot$. Indeed, for $\tF\in\tE^\bot$ one computes
\[ \Hom^*(\pi_*(\tE),\pi_*(\tF)) \cong \Hom^*(\pi^*\pi_*(\tE),\tF) \cong
   \Hom^*(\tE \oplus \tau^*(\tE), \tF) =0 , \]
because $\tau^*(\tE)\cong\tE$. On the other hand, we can check that for any $C \in (\pi_*\tE)^\bot$ we have that $C\oplus C\otimes \omega_X \in \pi_*(\tE^\bot)$. Hence, $\Phi_i(C)\cong C$ or $\Phi_i(C)\cong C\otimes \omega_X$.
\end{remark}

Using the above, we get

\begin{thm}\label{thm:enriques}
Given a K3 surface $\tX$ with a fixed point free involution $\tau$, so that $\tX/\tau$ is an Enriques surface $X$, and a $\tau$-invariant spherical object $\tE$ in $\D(\tX)$, the spherical twist $\TT_\tE$ induces two equivalences $\Phi_1$, $\Phi_2$ of $\D(X)$ and hence there exist induced equivalences $\Phi_i^{[n]}$ of $\Xn$. The lifts of the functors $\Phi_i^{[n]}$ to the covering Calabi--Yau manifold $Z$ are exotic autoequivalences.
\end{thm}

\begin{proof}
First note that for the known examples of spherical objects $\tE\in\D(\tX)$, the right orthogonal $\tE^\bot$ is never empty. Hence, we see that the functors $\Phi_i$ are neither shift nor the identity and behave like a spherical functor. They can be lifted to $\D(\Xn)$ by Equation \ref{liftingtohilb}. Consider $n=2$ for simplicity and set $\Phi=\Phi_1$. We have seen above that $\Phi(\pi_*(\tE))\cong \pi_*(\tE)[-1]$ and, if $\tE^\bot \neq \emptyset$ and contains, say $\tF$, then $\Phi\pi_*(\tF)\cong \pi_*(\tF)$. Thus, $f(\Phi)$ applied to the object $\pi_*(\tE)\boxtimes \pi_*(\tE)$ is the shift by $-2$, whereas if we apply it to $\pi_*(\tF)\boxtimes \pi_*(\tF)$ we get the identity. Hence this equivalence of $\D(X^{[2]})$ is not in the group generated by shifts, automorphisms and line bundle twists. Similar arguments work for $n\neq 2$.
\end{proof}

\section{New autoequivalences on Hilbert schemes of K3 surfaces}

Let $X$ be a K3 surface and let $E\in \D(X)$ be a spherical object. Then $E^{\boxtimes n}\in \D(X^n)$ is
$\Sn$-invariant and simple, as follows from the K\"unneth formula:
 $\End(E^{\boxtimes n})=\End(E)^{\otimes n}=\kk$.
Hence, from the general theory of Section~\ref{sub:linearisations}, $E^{\boxtimes n}$ is
$\Sn$-linearisable in two ways. In this particular case, the following chain of
canonical isomorphisms ($\sigma\in \Sn$, considered both a permutation as well
as an automorphism of $X^n$) provides one linearisation:
\begin{align*}
       \sigma^*(E^{\boxtimes n})
&=     \sigma^*(\pi_1^*E\otimes\cdots\otimes\pi_n^*E) \\
&\cong \sigma^*\pi_1^*E\otimes\cdots\otimes\sigma^*\pi_n^*E \\
&\cong (\pi_1\sigma)^*E\otimes\cdots\otimes(\pi_n\sigma)^*E \\
&\cong \pi_{\sigma_1}^*E\otimes\cdots\otimes\pi_{\sigma_n}^*E \\
&\cong \pi_1^*E\otimes\cdots\otimes\pi_n^*E = E^{\boxtimes n}.
\end{align*}
We denote this linearised object by $E^{[n]}\in \D^\Sn(X^n)$. The other one is
given by $\lambda_\sigma=-1$ for all odd permutations (which forces
$\lambda_\sigma=1$ for the even ones); we denote this linearised object by
$E^{-[n]}\in\D^\Sn(X^n)$.
Note that the same arguments apply to the kernel $K$
of an autoequivalence of $\D(X)$, since any such object is simple by \cite[Lem.\ 4]{Ploog}

\begin{lem}
The objects $E^{[n]}$ and $E^{-[n]}$ are not isomorphic.
\end{lem}

\begin{proof}
Assume, for simplicity, that $n=2$, the other cases being analogous. Any morphism
$f\colon E^{[2]}\to E^{-[2]}$ is given by a map
$f\colon E^{\boxtimes 2}\to E^{\boxtimes 2}$ (i.e.\ a scalar $f\in\kk$) such
that $f\circ\lambda_\sigma^{E^{[2]}} = \lambda_\sigma^{E^{-[2]}}\circ f$
(where $\sigma=(1~2)$), i.e.\ $f=-f$, so that there is no isomorphism.

Alternatively, just use that a simple $S_n$-invariant object has two distinct linearisations, see Section \ref{sub:the varieties}.
\end{proof}

\begin{lem}
Let $E\in\D(X)$ be a spherical object (which is the same as a
$\IP^1$-object). Then $E^{[n]}$ and $E^{-[n]}$ are $\IP^n$-objects
in $\D^\Sn(X^n)$.
\end{lem}

\begin{proof}
We are looking at the full endomorphism algebras:
\begin{align*}
   \Hom^*_{\D^\Sn(X^n)}(E^{\pm[n]},E^{\pm[n]})
&= \Hom^*_{X^n}(E^{\boxtimes n},E^{\boxtimes n})^\Sn \\
&= (\Hom^*_X(E,E)^{\otimes n})^\Sn \\
&= (\kk[h_1]/h_1^2 \otimes\cdots\otimes\kk[h_n]/h_n^2)^\Sn
    \text{ with $\deg(h_i)=2$} \\
&= \kk[h]/h^{2n} \text{ with $h=h_1+\cdots+h_n$}
\end{align*}
where the first equality is the definition of morphisms in $\D^\Sn(X^n)$ and
the second is the K\"unneth formula.

Note that the $\Sn$-action on the Hom-space in the first line does not depend
on the choice of linearisation. For simplicity consider only the $n=2$ case. Given a morphism
$F\colon E\boxtimes E\to E\boxtimes E[m]$ for some $m\in\IZ$, we know from the K\"unneth formula that it
is of the form $F=\sum_i f_i\boxtimes g_i$ with $f_i\colon E\to E[k_i]$,
$g_i\colon E\to E[l_i]$ and $m=k_i+l_i$.
Now $F$ is an $\Sn$-invariant map if and only if
 $\sigma_*(F)\circ\lambda=\lambda\circ F$
(where $\sigma=(1~2)$ as before and $\lambda\coloneq\lambda_\sigma^{\pm[2]}=\pm1$).
Note that $\sigma_*(F)=\sum_i g_i\boxtimes f_i$. Since $\lambda$ is a scalar,
it commutes anyway, and the condition is just
$\sum_i g_i\boxtimes f_i=\sum_i f_i\boxtimes g_i$.
As, up to constants, $\End^\bullet(E)$ only has the maps $\id$ and $h\colon E\to E[2]$, this forces $m\in\{0,2,4\}$ and
 $F=\sum_{\tau\in \PP_{2,m/2}}
       \pi_{\tau_1}^*h\otimes\cdots\otimes\pi_{\tau_{m/2}}^*h$, where $\PP_{n,k}$ is the set of $k$-element subsets of $\{1,\dots,n\}$. The
argument generalises to $n>2$.
\end{proof}

\begin{cor}\label{orthogonal}
We have $E^{-[n]} \in (E^{[n]})^\bot$.
\end{cor}

\begin{proof}
A simple application of the previous two lemmas.
\end{proof}



\begin{remark}
While it is true that $\ko_{X}^{-[n]}$ and $\ko_{X}^{[n]}$ are line bundle
objects on $\D^\Sn(X^n)$, and the square of the former is the latter
 ($\ko_{X}^{-[n]}\otimes\ko_{X}^{-[n]}=\ko_{X}^{[n]}$),
this does not contradict the fact that $\Pic(\Xn)$ has no torsion: the
derived McKay equivalence $\Phi\colon\D^\Sn(X^n)\isom\D(\Xn)$ does not respect
tensor products. Using the universal cycle as Fourier--Mukai kernel for $\Phi$,
as is done in \cite{BKR}, it is easy to check that $\Phi$ maps
$\ko_{X}^{[n]}$ to $\ko_{X^{[n]}}$ and $\ko_{X}^{-[n]}$ to $\ko_{X^{[n]}}(D)$, where
$D$ is the exceptional divisor.
\end{remark}

Given a spherical object $E\in\D(X)$, we have four canonically associated
autoequivalences of $\D^\Sn(X^n)\cong\D(\Xn)$. Write $\kp$ for the kernel of the spherical twist $\TT_E$.

\begin{tabular}{l p{11.5cm}}
$\PP_{E^{[n]}}$  & using the $\IP^n$-object $E^{[n]}\in\D^\Sn(X^n)$ \\
$\PP_{E^{-[n]}}$ & using the $\IP^n$-object $E^{-[n]}\in\D^\Sn(X^n)$ \\
$\TT_E^{\;[n]}$  & inducing $\TT_E$ with the constant diagonal
                    $\Sn$-linearisation on $\kp^{\boxtimes n}$\\
$\TT_E^{\;-[n]}$ & inducing $\TT_E$ with the other diagonal
                    $\Sn$-linearisation on $\kp^{\boxtimes n}$
\end{tabular}

The values of these functors on the objects $E^{\pm[n]}$ are crucial:

\begin{lem}\label{lem:value-equiv}
The following equalities hold.
\[ \begin{array}{l rcl @{,\qquad} rcl}
 \PP_{E^{[n]}}\colon   & E^{[n]} &\mapsto & E^{[n]}[-2n]   & E^{-[n]} &\mapsto & E^{-[n]} \\
 \PP_{E^{-[n]}}\colon  & E^{[n]} &\mapsto & E^{[n]}       & E^{-[n]} &\mapsto & E^{-[n]}[-2n] \\
 \TT_E^{\;[n]}\colon  & E^{[n]} &\mapsto & E^{[n]}[-n]   & E^{-[n]} &\mapsto & E^{-[n]}[-n]  \\
 \TT_E^{\;-[n]}\colon & E^{[n]} &\mapsto & E^{[n]}[-n]   & E^{-[n]} &\mapsto & E^{-[n]}[-n]
\end{array} \]
\end{lem}

\begin{proof}
The first two lines are general statements for $\IP^n$-twists.

For the other ones, we use that the induced functors
$\TT_E^{\;\pm[n]}$ are computed like ordinary FM transforms, linearising
the end result in the way explained in Section \ref{sub:Hilbert}. In particular, all four objects
 $\TT_E^{\;\pm[n]}(E^{\pm[n]})\in\D^\Sn(X^n)$
have the same underlying object $E^{\boxtimes n}[-n]\in\D(X^n)$. Indeed, if we ignore linearisations $\TT_E^{\;\pm[n]}$ is just $\TT_E^{\;\times n}$; since $\TT_E(E)=E[-1]$ we get
  $\TT_E^{\;\times n}(E^{\boxtimes n}) =
   (E[-1])^{\boxtimes n} = E^{\boxtimes n}[-n]$.

Since the latter has only two different linearisations, we find that
$\TT_E^{\;[n]}$ and $\TT_E^{\;-[n]}$ will (up to shift) permute the set
$\{E^{[n]},E^{\pm[n]}\}$.

The diagonal linearisation on $\kp^{\pm[n]}\otimes\pi_1^*E^{-[n]}$ is
given by the (tensor) product of the two compound linearisations. The
pushforward
\[\pi_{2*}(\kp^{\pm[n]}\otimes\pi_1^*E^{\pm[n]})\cong\TT_E^{\pm[n]}(E^{\pm[n]})\]
then inherits the linearisation of that sign.
\end{proof}

\begin{remark}
Note that $\TT_E^2=\PP_E$. Since $\deltan$ is a group homomorphism, we can immediately compute the values of $\PP_E^{\;\pm[n]}$ on the above mentioned four objects.
\end{remark}

\begin{thm}\label{thm:hilb}
The autoequivalences $\TT_E^{\;[n]}$ and $\TT_E^{\;-[n]}$ are exotic and distinct.
\end{thm}

\begin{proof}
The functors are distinct since the kernel of $\TT_E^{\;[n]}$ is the tensor product with itself of the kernel of $\TT_E^{\;-[n]}$. It follows immediately from Lemma \ref{lem:value-equiv} that these functors can neither be spherical nor $\IP^n$-twists. Taking an object $F \in E^\bot$ (which exists for all known spherical objects $E$) and arguing as in the lemma, we see that $\TT_E^{\;[n]}$ and $\TT_E^{\;-[n]}$ shift some objects and leave others invariant, hence they are not contained in the group of standard equivalences. 
\end{proof}

\begin{remark}
Again, it follows from Lemma~\ref{lem:value-equiv} that the four autoequivalences $\PP_{E^{\pm[n]}}$, $\PP_E^{\pm[n]}$ are distinct. However, it is not yet clear that $\PP_E^{\pm[n]}$ are not $\IP^n$-twists associated to other $\IP^n$-objects. This seems unlikely, especially for the $\IP^n$-object $\ko_{X^{[n]}}$.

Furthermore, the autoequivalences $\TT_E^{\;[n]} \TT_E^{\;-[n]}$ and $\TT_E^{\;-[n]} \TT_E^{\;[n]}$ are potentially also exotic.
\end{remark}

\end{document}